\documentclass[
final
]{dmtcs-episciences}
\usepackage[margin=3cm]{geometry}
\usepackage{amsthm}
\usepackage{amssymb}
\usepackage{amsmath}
\usepackage{xcolor} 
\usepackage[style=base]{subcaption}
\usepackage[round]{natbib}

\usepackage[noend]{algpseudocode}
\usepackage[linesnumbered,ruled]{algorithm2e}
\usepackage{hyperref, verbatim}
\usepackage{mathtools}

\theoremstyle{definition}
\newtheorem{theorem}{Theorem}
\newtheorem{definition}[theorem]{Definition}
\newtheorem{example}[theorem]{Example}
\newtheorem{question}[theorem]{Question}

\newtheorem{lemma}[theorem]{Lemma}
\newtheorem{conjecture}[theorem]{Conjecture}

\newcommand{\ZZ}{\mathbb{Z}}
\newcommand{\NN}{\mathbb{N}}

\input xy	
\xyoption{all}

\usepackage[most]{tcolorbox}
\usetikzlibrary{shadows}
\definecolor{cqcqcq}{rgb}{0.7529411764705882,0.7529411764705882,0.7529411764705882}
\definecolor{dtsfsf}{rgb}{0.8274509803921568,0.1843137254901961,0.1843137254901961}
\definecolor{ffwwzz}{rgb}{1,0.4,0.6}
\definecolor{ffzzcc}{rgb}{1,0.6,0.8}

\begin{document}

\title{Bounds on $(t,r)$ broadcast domination of $n$-dimensional grids}

\author{Tom Shlomi \affiliationmark{1}}
\affiliation{Harvard University} 
\received{2019-08-30}

\revised{2020-02-28, 2020-03-18, 2022-03-19}

\accepted{2022-03-31}

\publicationdetails{24}{2022}{1}{14}{5732}
 
\maketitle

\begin{abstract}
In this paper, we study a variant of graph domination known as $(t, r)$ broadcast domination, first defined in Blessing et al. (2015). In this variant, each broadcast provides $t-d$ reception to each vertex a distance $d < t$ from the broadcast. If $d \ge t$ then no reception is provided. A vertex is considered dominated if it receives $r$ total reception from all broadcasts. Our main results provide some upper and lower bounds on the density of a $(t, r)$ dominating pattern of an infinite grid, as well as methods of computing them. Also, when $r \ge 2$ we describe a family of counterexamples to a generalization of Vizing's Conjecture to $(t,r)$ broadcast domination.\end{abstract}


\section{Introduction}

Let $G=(V,E)$ be a graph with vertex set $V$ and edge set $E$.  A \emph{dominating set} $D$ of $G$ is any subset of $V$ for which every vertex in $G$ is either contained in $D$ or is adjacent to a vertex in $D$. 
The \emph{domination number}, $\gamma(G)$, is defined to be the size of the smallest possible dominating set of $G$. 

There have been over 2000 papers written on domination theory and 80 different domination related parameters defined on graphs in the past 50 years, and we refer the interested reader to the survey text written by Haynes, Hedetniemi, and Slater for a comprehensive introduction to other domination parameters \cite{haynes2013fundamentals}. In this paper, we focus primarily on $(t,r)$ broadcast domination, which was defined by Blessing, Insko, Johnson, and Mauretour in 2015 \cite{blessing2015t,farina2016new} as a generalization of domination and distance domination. Following their conventions, a \emph{$(t,r)$ broadcast domination set} is defined as follows: Consider a set $D \subseteq V$, fix a desired transmission strength $t \in \NN = \{1,2,3, ...\}$, and fix a desired reception strength $r \in \NN$, with $t \geq r$.  Define the \emph{transmission neighborhood} $N_t(v)$ of a vertex $v$ in $V$ to be the set of points a distance less than or equal to $t$ from $v$. Then define the reception $r(v)$ at each vertex $v \in V$ to be 
\begin{displaymath}
r(v) := \sum_{d \in D \cap N_t(v)}t - dis(d, v)
\end{displaymath}
where $dis(v_1, v_2)$ is the distance between two vertices of a graph. The set $D$ is a $(t, r)$ dominating set if every vertex has a reception of at least $r$, i.e. for all vertices $v$, $r(v) \geq r$.
The \emph{$(t,r)$ broadcast domination number}, $\gamma_{t,r}(G)$, is defined to be the smallest cardinality of any $(t,r)$ broadcast dominating set for $G$. Section \ref{subsec:ex} has an example where $t=3$ and $r=2$. 

Regular graph domination has numerous applications in computer science and other fields. While they are not explored in this paper, there are many potential applications for $(t,r)$ broadcast domination as well. For example, it could be useful for municipal planning. In a city, every house needs to be within a certain radius of a fire station.  However, it is possible that having two fire stations just outside that radius is acceptable, since it increases the chance that a fire truck from one of them could make it on time. This scenario could be approximately modeled by representing the city as a graph and the fire stations as $(t,r)$ broadcasts.

In their paper, Blessing, Insko, Johnson, and Mauretour established exact numbers for the $(t,r)$ broadcast domination number of small grids, as well as bounds for larger ones \cite{blessing2015t}. Drews, Harris, and Randolph expanded this work by computing $(t,r)$ broadcast domination patterns of infinite grids and finding their densities \cite{drews2019optimal}.

While most of our work is done on grids, which are Cartesian products of paths, we also explore Cartesian products in more generality. In particular, we look at a generalization of Vizing's Conjecture, perhaps the most famous conjecture in domination theory. This conjecture was posed by Vizing in 1968 and states that the domination number of a Cartesian product of graphs is at least as large as the product of their domination numbers.
\begin{conjecture} [Vizing \cite{Viz68}]
Let $G$ and $H$ be finite graphs, and let $G \Box H$ be their Cartesian product. Then
\[ \gamma (G \Box H)  \geq  \gamma (G) \gamma(H).\]
\end{conjecture}
    
While this conjecture remains unresolved, there have been many special cases of the conjecture proven in the past 50 years as detailed in the survey article by Bre\v{s}ar, Dorbec, Goddard, Hartnell, Henning, Klav\v{z}ar, and Rall \cite{BreDorGodHarHenKlaRal12}.
Since $(t,r)$ broadcast domination generalizes domination and distance domination, it seems natural to ask whether an analog of Vizing's conjecture holds for $(t,r)$ broadcast domination.

In Section \ref{sec:vizing} of this paper, we show that this is not the case when $r>1$ by providing a family of counterexamples. However, we conjecture that it holds for $r=1$, which is a strengthened form of Vizing's Conjecture.

We then move on from general Cartesian products, and focus on grid graphs of arbitrary dimensions $\ZZ^d$, which is the Cartesian product of $d$ infinite paths. In Section \ref{sec:ball}, we calculate $N_t(v)$ for any vertex $v$ in an $n$-dimensional hypercubic lattice. We give both combinatorial formulas and generating functions, and explore a connection with the Delannoy numbers. The $(m,n)$th Delannoy number counts the number of lattice paths to an integer point $(m, n)$ with only the steps $(1,0)$, $(0,1)$, and $(1,1)$. This connection was previously found by Zaitsev in 2017 \cite{Zait17}. In Theorem~\ref{thm:trivial} we find a surprising bijection between points in $\ZZ^n$ that are within distance $d$ of the origin and points in $\ZZ^d$ that are within distance $n$ of the origin.  

In Section \ref{sec:lbound}, we derive a lower bound on the minimum density of a $(t, r)$ broadcast dominating set of an infinite grid of any dimension. From that, we derive a lower bound on $(t,r)$ broadcast domination number of a finite grid.

We conclude by describing an algorithm which generates upper bounds on the density $(t, r)$ dominating set of a $2$ or $3$ dimensional infinite grid. The algorithm is similar to the one described in a recent paper by Drews, Harris, and Randolph \cite{drews2019optimal}. However, our algorithm works in more dimensions and we show it is more computationally efficient.

\subsection{Examples}
\label{subsec:ex}

Figure \ref{examplefig} illustrates examples of distinct sets of $(3, 2)$ broadcasts in a $5 \times 5$ grid. Black vertices have a reception of at least $3$ (since they are a distance $0$ away from a broadcast), red vertices have a reception of at least $2$ (since they have a distance $1$ from a broadcast), pink vertices have a reception of at least $2$ (since they have a distance $2$ from at least two broadcasts), gray vertices have a reception of 1 (since they have a distance 2 from exactly 1 broadcast), and white vertices have a reception of $0$ (since they have a distance of at least $3$ from all broadcasts). For a grid to be $(3, 2)$ dominated by its black vertices, there must be no gray or white vertices. Thus, Figure 1(A) is fully dominated while Figure 1(B) is not.
\begin{figure}[ht!]
\centering
\begin{subfigure}[b]{0.4\textwidth}
\begin{tikzpicture}
 \tikzstyle{ghost node}=[draw=none]
\tikzset{white node/.style={circle,draw=black, inner sep=2.5}}
\tikzset{red node/.style={circle,draw=black, fill=red, inner sep=2.5}}
\tikzset{black node/.style={circle,draw=black, fill=black, inner sep=2.5}}
\tikzset{pink node/.style={circle,draw=black, fill=pink, inner sep=2.5}}
 \foreach \x in {1, 2, 3,4,5}
    \foreach \y in {1, 2, 3,4,5} 
        \node [pink node] (\x\y) at (1*\x,1*\y){};
       \node[label={[xshift=0.15cm, yshift=0cm]1+1}] [pink node] (11) at (1*1,1*1){};
       \node[label={[xshift=0.15cm, yshift=0cm]1+1}] [pink node] (15) at (1*1,1*5){};
       \node[label={[xshift=0.15cm, yshift=0cm]1+1}] [pink node] (22) at (1*2,1*2){};
       \node[label={[xshift=0.15cm, yshift=0cm]1+1}] [pink node] (24) at (1*2,1*4){};
       \node[label={[xshift=0.0cm, yshift=0cm]1+1+1+1}] [pink node] (33) at (1*3,1*3){};
       \node[label={[xshift=0.15cm, yshift=0cm]1+1}] [pink node] (44) at (1*4,1*4){};
       \node[label={[xshift=0.15cm, yshift=0cm]1+1}] [pink node] (42) at (1*4,1*2){};
       \node[label={[xshift=0.15cm, yshift=0cm]1+1}] [pink node] (55) at (1*5,1*5){};
       \node[label={[xshift=0.15cm, yshift=0cm]1+1}] [pink node] (51) at (1*5,1*1){};
       \node[label={[xshift=0.15cm, yshift=0cm]3}] [black node] (13) at (1*1,1*3){};
       \node[label={[xshift=0.15cm, yshift=0cm]3}] [black node] (31) at (1*3,1*1){};
       \node[label={[xshift=0.15cm, yshift=0cm]3}] [black node] (35) at (1*3,1*5){};
       \node[label={[xshift=0.15cm, yshift=0cm]3}] [black node] (53) at (1*5,1*3){};
       \node[label={[xshift=0.15cm, yshift=0cm]2}] [red node] (12) at (1*1,1*2){};
       \node[label={[xshift=0.15cm, yshift=0cm]2}] [red node] (14) at (1*1,1*4){};
       \node[label={[xshift=-0.15cm, yshift=0cm]2}] [red node] (23) at (1*2,1*3){};
       \node[label={[xshift=0.15cm, yshift=0cm]2}] [red node] (21) at (1*2,1*1){};
       \node[label={[xshift=0.15cm, yshift=0cm]2}] [red node] (25) at (1*2,1*5){};
       \node[label={[xshift=0.15cm, yshift=0cm]2}] [red node] (32) at (1*3,1*2){};
       \node[label={[xshift=0.15cm, yshift=0cm]2}] [red node] (34) at (1*3,1*4){};
       \node[label={[xshift=0.15cm, yshift=0cm]2}] [red node] (41) at (1*4,1*1){};
       \node[label={[xshift=0.15cm, yshift=0cm]2}] [red node] (43) at (1*4,1*3){};
       \node[label={[xshift=0.15cm, yshift=0cm]2}] [red node] (45) at (1*4,1*5){};
       \node[label={[xshift=0.15cm, yshift=0cm]2}] [red node] (52) at (1*5,1*2){};
       \node[label={[xshift=0.15cm, yshift=0cm]2}] [red node] (54) at (1*5,1*4){};
       
  \foreach \x in {1,2,3,4,5}
    \foreach \y  [count=\yi from 2] in {1,2,3,4} 
      \path[] (\x\y)edge(\x\yi)(\y\x)edge(\yi\x);
\end{tikzpicture} 
\caption{}
\label{efa1}
\end{subfigure}
\centering
\begin{subfigure}[b]{0.4\textwidth}

\begin{tikzpicture}
 \tikzstyle{ghost node}=[draw=none]
\tikzset{white node/.style={circle,draw=black, fill = white, inner sep=2.5}}
\tikzset{red node/.style={circle,draw=black, fill=red, inner sep=2.5}}
\tikzset{black node/.style={circle,draw=black, fill=black, inner sep=2.5}}
\tikzset{pink node/.style={circle,draw=black, fill=pink, inner sep=2.5}}
\tikzset{gray node/.style={circle,draw=black, fill=gray, inner sep=2.5}}

 \foreach \x in {1, 2, 3,4,5}
    \foreach \y in {1, 2, 3,4,5} 
       \node [white node] (\x\y) at (1*\x,1*\y){};
      \node[label={[xshift=0.15cm, yshift=0cm]1+1}] [pink node] (13) at (1*1,1*3){};
      \node[label={[xshift=0.15cm, yshift=0cm]1+1}] [pink node] (31) at (1*3,1*1){};
      \node[label={[xshift=0.15cm, yshift=0cm]1+1}] [pink node] (35) at (1*3,1*5){};
      \node[label={[xshift=0.15cm, yshift=0cm]1+1}] [pink node] (53) at (1*5,1*3){};
       \node[label={[xshift=0.15cm, yshift=0cm]3}] [black node] (11) at (1*1,1*1){};
       
       \node[label={[xshift=0.15cm, yshift=0cm]3}] [black node] (51) at (1*5,1*1){};
       \node[label={[xshift=0.15cm, yshift=0cm]3}] [black node] (15) at (1*1,1*5){};
       \node[label={[xshift=0.15cm, yshift=0cm]3}] [black node] (55) at (1*5,1*5){};
       \node[label={[xshift=0.15cm, yshift=0cm]2}] [red node] (12) at (1*1,1*2){};
       \node[label={[xshift=0.15cm, yshift=0cm]2}] [red node] (14) at (1*1,1*4){};
       \node[label={[xshift=0.15cm, yshift=0cm]1}] [gray node] (22) at (1*2,1*2){};
       \node[label={[xshift=0.15cm, yshift=0cm]2}] [red node] (21) at (1*2,1*1){};
       \node[label={[xshift=0.15cm, yshift=0cm]2}] [red node] (25) at (1*2,1*5){};
       \node[label={[xshift=0.15cm, yshift=0cm]1}] [gray node] (42) at (1*4,1*2){};
       \node[label={[xshift=0.15cm, yshift=0cm]1}] [gray node] (44) at (1*4,1*4){};
       \node[label={[xshift=0.15cm, yshift=0cm]2}] [red node] (41) at (1*4,1*1){};
       \node[label={[xshift=0.15cm, yshift=0cm]1}] [gray node] (24) at (1*2,1*4){};
       \node[label={[xshift=0.15cm, yshift=0cm]2}] [red node] (45) at (1*4,1*5){};
       \node[label={[xshift=0.15cm, yshift=0cm]2}] [red node] (52) at (1*5,1*2){};
       \node[label={[xshift=0.15cm, yshift=0cm]2}] [red node] (54) at (1*5,1*4){};

  \foreach \x in {1,2,3,4,5}
    \foreach \y  [count=\yi from 2] in {1,2,3,4} 
      \path[] (\x\y)edge(\x\yi)(\y\x)edge(\yi\x);

\end{tikzpicture}
\caption{}
\label{efb1}
\end{subfigure}

\caption{A (3,2) broadcast dominating set of $P_5 \Box P_5$ and a (3,2) non-dominating broadcast set of $P_5 \Box P_5$.}
\label{examplefig}
\end{figure}
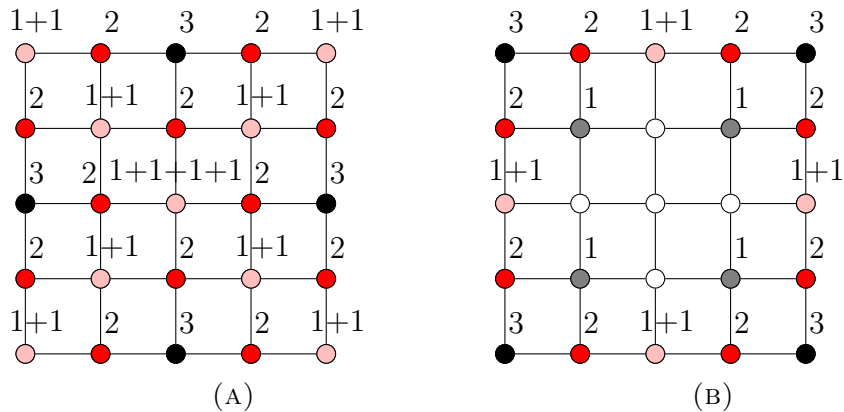

\section{The \texorpdfstring{$(t, r)$}{(t, r)} analog of Vizing's Conjecture}  \label{sec:vizing}
Since $(t,r)$ broadcast domination theory generalizes domination theory and distance domination theory, it is reasonable to ask whether Vizing's conjecture applies more generally to $(t,r)$ broadcast domination.

Take graphs $G = \{V_G, E_G\}, H = \{V_H, E_H\}$. Let $V$ be the set Cartesian product of $V_G$ and $V_H$. Let $E \subset V \times V$ contain  $((g_1, h_1), (g_2, h_2))$, with $g_1, g_2 \in V_G$ and $h_1, h_2 \in V_H$, if $g_1 = g_2$ and $h_1$ is adjacent to $h_2$ or $h_1 = h_2$ and $g_1$ is adjacent to $g_2$. We say that $\{V, E\}$ is the graph Cartesian product of $G$ and $H$, represented as $G \Box H$.
\begin{question} \label{question:1}
Let $\gamma_{t,r}(G)$ be the $(t,r)$ broadcast domination number of $G$. If $G$ and $H$ are finite graphs, is it true that
\begin{displaymath}
\gamma_{t,r}(G \Box H) \ge \gamma_{t,r}(G)\gamma_{t,r}(H)\end{displaymath}
 for all $t, r \in \ZZ$ such that $t \ge r \ge 1$?  
\end{question}

The main result of this section shows that this is not the case when $r>1$.
Before constructing a family of counterexamples to the 
$(t,r)$ broadcast domination analog of Vizing's conjecture, we first prove a result describing the $(t,r)$ broadcast domination numbers of cyclic graphs.

\begin{lemma}\label{lemma:counterexample}
Let $C_n$ denote the cycle graph with $n$ nodes. 
Then the $(t, r)$ broadcast domination number of $C_{2(t-r+1)}$ is 2.
\end{lemma}

\begin{proof} 
Let $n = 2(t-r + 1)$. Label the vertices of $C_n$ as $(0, 1, 2, \ldots , n-1)$, going clockwise. Suppose without loss of generality that a broadcast of strength $t$ is placed at vertex $(0)$.  Then vertex $(\frac{n}{2})$ has reception $r-1$, because it is distance $t - r + 1$ from vertex $(0)$, and so it is not yet dominated. 
Since vertex $(\frac{n}{2})$ is the only non-dominated vertex in $C_n$ if a broadcast is placed at vertex $(0)$,  every vertex can be dominated if another broadcast is added at $(\frac{n}{2})$. Hence $\gamma_{t,r}(C_{2(t-r+1)})=2.$
\end{proof} 

 The following result establishes a family of counterexamples to the inequality in  Question \ref{question:1}.
\begin{theorem}\label{thm:CC}
Let $C_n$ be a cycle graph with $n= 2(t-r+1)$. 
If $r \ge 2$, then the $(t, r)$ broadcast domination number of $C_n \Box C_n$ is 
$  \gamma_{t,r}(C_n \Box C_n) = 2$.
\end{theorem}

\begin{proof}
Let $G = C_n \Box C_n$.  Place a broadcast at $(0, 0)$ and at $(\frac{n}{2}, \frac{n}{2})$. Note that this follows the same numbering used to prove Lemma \ref{lemma:counterexample}. Since $(0)$ and $(\frac{n}{2})$ have the maximum possible distance between them in $C_n$ and since distance in $G$ is the sum of the distance between the two pairs of coordinates, the vertices $(0, 0)$ and $(\frac{n}{2}, \frac{n}{2})$ must have the maximum possible distance in $G$. As a result, for any vertex $w$ in $G$, the sum of the distance from $w$ to $(0,0)$ and from $w$ to $(\frac{n}{2}, \frac{n}{2})$ is $n$. The reception at $w$ is $t-d((0,0),w)+t-d((\frac{n}{2}, \frac{n}{2}),w) = 2t - n = 2r-2$. Since we assume $r \ge 2$, the reception at $w$ is greater than $r$, so $\gamma_{t,r}(G) \le 2$.

Without loss of generality, assume $G$ can be dominated by one broadcast at $(0,0)$. The vertex $(\frac{n}{2},0)$ is a distance $t - r + 1$ from the broadcast, so it only receives $r - 1$ reception and is undominated. Thus the $(t,r)$ broadcast domination number of $G$ cannot be 1, so $\gamma_{t,r}(G) = 2$.
\end{proof}

Combining the results from Lemma~\ref{lemma:counterexample} and Theorem~\ref{thm:CC} shows that \begin{displaymath}\gamma_{t,r}(C_n \Box C_n) =2 <  \gamma_{t,r}(C_n)\gamma_{t,r}(C_n) = 4\end{displaymath} whenever $n= 2(t-r+1)$ and $r \ge 2$, and hence the analog of Vizing's Conjecture for $(t, r)$ broadcast domination posed in Question \ref{question:1} is false when $r\ge2$.  We illustrate this counterexample in Figure \ref{fig:C4} for the case $(t,r) = (3,2)$.
\begin{figure}[ht!]
\centering
\begin{subfigure}[b]{0.4\textwidth}
\begin{tikzpicture}
 \tikzstyle{white node}=[draw=none]
\tikzset{red node/.style={circle,draw=black, fill=red, inner sep=2.5}}
\tikzset{black node/.style={circle,draw=black, fill=black, inner sep=2.5}}
\tikzset{pink node/.style={circle,draw=black, fill=pink, inner sep=2.5}}
\tikzset{gray node/.style={circle,draw=black, fill=gray, inner sep=2.5}}

       \node[label={[xshift=0.15cm, yshift=0cm]3}] [black node] (11) at (1*1,1*1){};
       \node[label={[xshift=0.15cm, yshift=0cm]2}] [red node] (15) at (1*1,1*5){};
       \node[label={[xshift=0.15cm, yshift=0cm]1}] [gray node] (55) at (1*5,1*5){};
       \node[label={[xshift=0.15cm, yshift=0cm]2}] [red node] (51) at (1*5,1*1){};
       \path[] (11)edge(51)(51)edge(55)(55)edge(15)(15)edge(11);
\end{tikzpicture} 
\caption{}
\label{efa2}
\end{subfigure}
\centering
\begin{subfigure}[b]{0.4\textwidth}

\begin{tikzpicture}
 \tikzstyle{ghost node}=[draw=none]
\tikzset{white node/.style={circle,draw=black, fill = white, inner sep=2.5}}
\tikzset{red node/.style={circle,draw=black, fill=red, inner sep=2.5}}
\tikzset{black node/.style={circle,draw=black, fill=black, inner sep=2.5}}
\tikzset{pink node/.style={circle,draw=black, fill=pink, inner sep=2.5}}
\tikzset{gray node/.style={circle,draw=black, fill=gray, inner sep=2.5}}

 \foreach \x in {0,1, 2, 3,4,5}
    \foreach \y in {0,1, 2, 3,4,5} 
       \node [ghost node] (\x\y) at (1*\x,1*\y){};
      \node[label={[xshift=0.15cm, yshift=0cm]1+1}] [pink node] (13) at (1*1,1*3){};
      \node[label={[xshift=0.15cm, yshift=0cm]1+1}] [pink node] (31) at (1*3,1*1){};
       \node[label={[xshift=0.15cm, yshift=0cm]1+1}] [pink node] (11) at (1*1,1*1){};
       \node[label={[xshift=0.15cm, yshift=0cm]2}] [red node] (12) at (1*1,1*2){};
       \node[label={[xshift=0.15cm, yshift=0cm]2}] [red node] (14) at (1*1,1*4){};
       \node[label={[xshift=0.15cm, yshift=0cm]3}] [black node] (22) at (1*2,1*2){};
       \node[label={[xshift=0.15cm, yshift=0cm]2}] [red node] (21) at (1*2,1*1){};
       \node[label={[xshift=0.15cm, yshift=0cm]1+1}] [pink node] (42) at (1*4,1*2){};
       \node[label={[xshift=0.15cm, yshift=0cm]3}] [black node] (44) at (1*4,1*4){};
       \node[label={[xshift=0.15cm, yshift=0cm]2}] [red node] (41) at (1*4,1*1){};
       \node[label={[xshift=0.15cm, yshift=0cm]1+1}] [pink node] (24) at (1*2,1*4){};
       \node[label={[xshift=0.15cm, yshift=0cm]2}] [red node] (32) at (1*3,1*2){};
       \node[label={[xshift=0.15cm, yshift=0cm]2}] [red node] (23) at (1*2,1*3){};
       \node[label={[xshift=0.15cm, yshift=0cm]1+1}] [pink node] (33) at (1*3,1*3){};
       \node[label={[xshift=0.15cm, yshift=0cm]2}] [red node] (34) at (1*3,1*4){};
       \node[label={[xshift=0.15cm, yshift=0cm]2}] [red node] (43) at (1*4,1*3){};

  \foreach \x in {1,2,3,4}
    \foreach \y  [count=\yi from 1] in {0,1,2,3,4} 
      \path[] (\x\y)edge(\x\yi)(\y\x)edge(\yi\x);

\end{tikzpicture}
\caption{}
\label{efb2}
\end{subfigure}

     \caption{Under $(3,2)$ broadcast domination, $C_4$ is not dominated by 1 broadcast, so $\gamma_{3,2}(C_4) = 2$, and  $C_4 \Box C_4$ is dominated by 2 broadcasts, so $\gamma_{3,2}(C_4 \Box C_4) = 2$.}
    \label{fig:C4}
\end{figure}
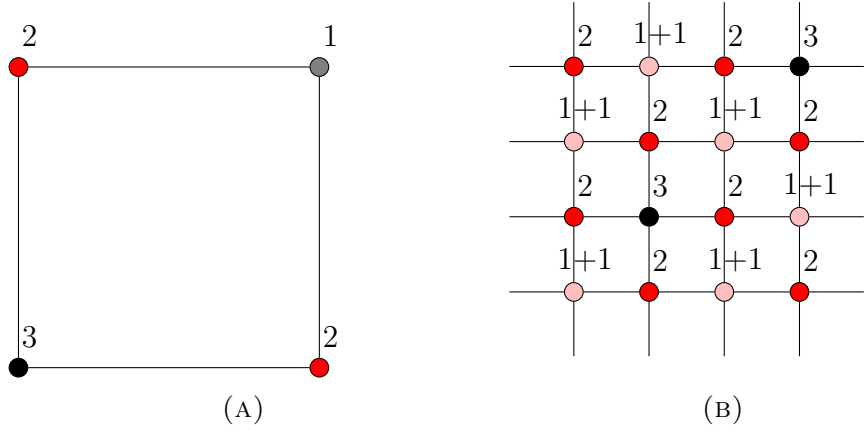

We have shown that the $(t,r)$ broadcast domination analog of Vizing's conjecture does not hold when $r>1$. However, we have not discovered any counterexamples that violate its bound by more that a factor of 2. In addition, none of our counterexamples apply for $(t,1)$ broadcast domination, which coincides with $(t-1)$-distance domination. As a result we state the following generalizations of Vizing's conjecture.

\begin{conjecture}
\label{conj:hviz}
Let $\gamma_{t,r}(G)$ be the $(t,r)$ broadcast domination number of $G$. For all finite graphs $G$ and $H$
\begin{displaymath}\gamma_{t,r}(G \Box H) \ge \frac{1}{2}\gamma_{t,r}(G)\gamma_{t,1}(H)\end{displaymath}
\end{conjecture}
\begin{conjecture}  [Generalizing Vizing's Conjecture]
\label{conj:cgvizing}
Let $\gamma_{t,1}(G)$ be the $(t,1)$ broadcast domination (or $(t-1)$-distance domination) number of $G$. For all finite graph $G,H$
\begin{displaymath}\gamma_{t,1}(G \Box H) \ge \gamma_{t,1}(G)\gamma_{t,1}(H). \end{displaymath}
\end{conjecture}

Previous bounds on $(t,r)$ domination of paths, Theorem 11, and the results from Algorithm 1 all provide further evidence of these conjectures.

\section{Balls in Lattices}
    
\label{sec:ball}

The later sections in this paper study $(t,r)$ broadcast domination on grid graphs, which are the Cartesian products of a finite number of paths. These results rely on the number of vertices a distance $d$ away from a given vertex in a grid. Since that has no closed expression, we were motivated to study that and related functions in this section.

 Let $\ZZ^n$ denote an infinite $n$-dimensional $\ZZ$-lattice.  Let $(x_1, \ldots , x_n)$ denote a lattice point in $\ZZ^n$. Define $S_n(d)$ to be the set of vertices with graph distance $d$ from the origin in $\ZZ^n$. An example is shown in Figure \ref{fig:2}. Similarly, define $B_n(d)$ to be the set of vertices with graph distance less than or equal to $d$ from the origin. This section focuses on counting the number of vertices in $S_n(d)$ and $B_n(d)$. 
 Not only are these formulas of interest in their own right, but, in Section \ref{sec:lbound}, we will use these combinatorial formulas in deriving bounds on the density of dominating sets.

\begin{figure}[ht!]
\centering
\begin{tikzpicture}[scale =.8]
 \tikzstyle{ghost node}=[draw=none]
\tikzset{white node/.style={circle,draw=black, inner sep=2.5}}
\tikzset{red node/.style={circle,draw=black, fill=red, inner sep=2.5}}
\tikzset{black node/.style={circle,draw=black, fill=black, inner sep=2.5}}
\tikzset{pink node/.style={circle,draw=black, fill=pink, inner sep=2.5}}

\foreach \x in {1, 2, 3,4,5,6,7}
    \foreach \y in {1, 2, 3,4,5,6,7} 
        \node [white node] (\x\y) at (1*\x,1*\y){};
\node [red node] (44) at (1*4,1*4){};
\node [black node] (47) at (1*4,1*7){};
\node [black node] (56) at (1*5,1*6){};
\node [black node] (65) at (1*6,1*5){};
\node [black node] (74) at (1*7,1*4){};
\node [black node] (63) at (1*6,1*3){};
\node [black node] (52) at (1*5,1*2){};
\node [black node] (41) at (1*4,1*1){};
\node [black node] (32) at (1*3,1*2){};
\node [black node] (23) at (1*2,1*3){};
\node [black node] (14) at (1*1,1*4){};
\node [black node] (25) at (1*2,1*5){};
       \node [black node] (36) at (1*3,1*6){};
\foreach \x in {1,2,3,4,5,6,7}
    \foreach \y  [count=\yi from 2] in {1,2,3,4,5,6} 
        \path[] (\x\y)edge(\x\yi)(\y\x)edge(\yi\x);

\end{tikzpicture}  

\caption{The red vertex is the origin and the black vertices are the elements of $S_2(3)$. The number of black vertices is $\sum_{i = 0}^{1}\binom{2}{i}2^{2-i}\binom{2}{1 - i} = 12$.}
  \label{fig:2}

\begin{center}
\renewcommand{\arraystretch}{1.2}   
\begin{tabular}{|c|c| } \hline
 $n$ & $|S_n(d)|$ \\ \hline
     1  & 2 \\
     2 & $4d$   \\
     3 & $4d^2  + 2$   \\
     4 & $\frac{8}{3}(d^3+2d)$  \\ 
     5 & $\frac{2}{3}(2d^4 + 10d^2 + 3)$\\
     6 & $\frac{4}{15}(2d^5 + 20d^3 + 23d)$\\
     7 & $\frac{2}{45}(4d^6 + 70d^4 + 196d^2 + 45)$\\
     \hline
\end{tabular}
\end{center} 

\end{figure}
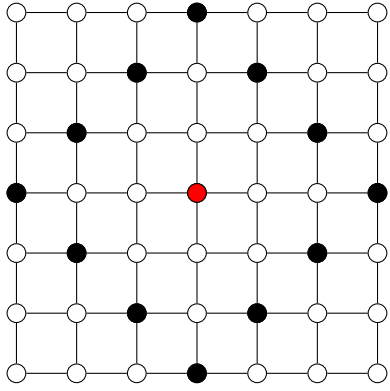

\begin{theorem}\label{theorem:shell}
[Theorem~1 \cite{Zait17}] When, $d \ge 1$, the number of vertices in $S_n(d)$ is given by 
\begin{displaymath}|S_n(d)| = \sum_{i = 0}^{n-1}\binom{n}{i}2^{n-i}\binom{d - 1}{n - i - 1}.\end{displaymath}
\end{theorem}
\begin{proof}
[\emph{This proof is independently derived, but similar to the proof given in \cite{Zait17}}]We partition \linebreak the set $S_n(d)$ into disjoint sets, indexed by how many coordinates are nonzero.  First, we count the vertices for which all of the coordinates are greater than 0. This the same as the problem of distributing $d$ identical items between $n$ containers such that each container has at least one item. There are $\binom{d - 1}{n - 1}$ such vertices. If some of the coordinates are negative instead of positive, the number remains the same. There are $2^n$ ways of assigning each coordinate to be negative or positive. So there are $2^n\binom{d - 1}{n - 1}$ vertices within distance $d$ away from the origin that have no coordinates that are $0$. 

Now assume $i$ of the coordinates are $0$. The sum of the magnitudes of the $n-i$ nonzero coordinates must be $d$. This creates $\binom{d - 1}{n - i - 1}$ possibilities. Then, there are $2^{n-i}$ ways to make each of the nonzero coordinates positive or negative. Finally, there are $\binom{n}{i}$ ways to choose which of the coordinates should be 0. Multiplying these numbers, we find that there are $\binom{n}{i}2^{n-i}\binom{d - 1}{n - i - 1}$ vertices in $S_n(d)$  that have exactly $i$ coordinates that are $0$. To get the total number of vertices in $S_n(d)$, we must sum over all the possible values of $i$. Since there must be some nonzero coordinate, $0 \leq i \leq n - 1$, we conclude that
$\displaystyle |S_n(d)| = \sum_{i = 0}^{n-1}\binom{n}{i}2^{n-i} \binom{d - 1}{n - i - 1} $.  \qedhere 
\end{proof}

In order to derive a generating function for $|S_n(d)|$ in Theorem~\ref{thm:S}, we first derive a generating function for $B_n(d)$. Let $B(x,y)$ be the multivariable generating function for $|B_n(d)|$, defined as 
\begin{displaymath}B(x,y) := \sum_{i=0}^{\infty}\sum_{j=0}^{\infty}|B_i(j)|x^iy^j.\end{displaymath} 
\begin{lemma}
Let $B_n(d)$ be the set of vertices a graph distance $d$ or less away from the origin in $\ZZ^n$. Then
\[|B_n(d)| = |B_{n-1}(d)| + |B_n(d-1)| + |B_{n-1}(d-1)|.\]
\label{lemma:recursion}
\end{lemma}
\begin{proof}
Partition $B_n(d)$ into three sets. The set of points a distance less than $d$ away, the set of points a distance of exactly $d$ away with an $x_1$ value of at least 0, and the set of points a distance of exactly $d$ away from the origin with an $x_1$ value of less than 0. The first set is $B_n(d-1)$. 

If we have a point a distance exactly $d$ away from the origin in $\ZZ^n$ with coordinates $(a_1, \ldots, a_n)$, we know that
\begin{displaymath}d = \sum_{i = 1}^n|a_i| \text{ and thus} \end{displaymath}\begin{displaymath} |a_1| = d - \sum_{i=2}^n|a_i|.\end{displaymath}
Since all the points in the second set have an $x_1$ value of at least 0, this means that we can uniquely determine their $x_1$ values from their other coordinates. The second set then is
\begin{displaymath}\left\{(a_1, \ldots, a_n)\bigg|\sum_{i=2}^n|a_i| \leq d, a_1 = d - \sum_{i=2}^n|a_i|\right\}.\end{displaymath}
We also have that $B_{n-1}(d)$ is
\begin{displaymath}\left\{(a_2, \ldots, a_n) \in \ZZ^n\bigg|\sum_{i=2}^n|a_i| \leq d \right\}.\end{displaymath}
Thus, there exists a bijection from $(a_1, \ldots, a_n)$ in the second set to $(a_2, \ldots,  a_n)$ in $B_{n-1}(d)$. 

Since all the points in the third set have an $x_1$ value of less than 0, this means that we can uniquely determine their $x_1$ values from their other coordinates. We then have that the third set is 
\begin{displaymath}\left\{(a_1, \ldots, a_n)\in \ZZ^n\bigg|\sum_{i=2}^n|a_i| < d, a_1 = -d + \sum_{i=2}^n|a_i|\right\}.\end{displaymath}
We also have that $B_{n-1}(d-1)$ is
\begin{displaymath}\left\{(a_2, \ldots, a_n)\bigg|\sum_{i=2}^n|a_i| < d\right\}.\end{displaymath}
Thus, there exists a bijection from $(a_1, \ldots, a_{n-1})$ in the third set to $(a_2,  \ldots, a_n)$ in $B_{n-1}(d-1)$.

We now have the recursion
$|B_n(d)| = |B_{n-1}(d)| + |B_n(d-1)| + |B_{n-1}(d-1)|.$ \qedhere
\end{proof}
\begin{theorem}\label{theorem:ball}
If $B(x,y)$ is the multivariable generating function for $|B_n(d)|$, then
\begin{displaymath}B(x, y) = \frac{1}{1-x-y-xy}.\end{displaymath}
\end{theorem}
\begin{proof}

Since $|B_0(0)| = 1$, we can use Lemma \ref{lemma:recursion} to derive

\begin{displaymath}B(x,y) = xB(x, y) + yB(x, y) + xyB(x, y) + 1.\end{displaymath} 

Through algebraic manipulation, we can get 

\begin{displaymath}B(x, y) = \frac{1}{1-x-y-xy}. \label{eq:B(x,y)} \qedhere\end{displaymath} 
\end{proof}

From Theorem~\ref{theorem:ball}, we can derive the generating function for $|S_n(d)|$.
\begin{theorem} \label{thm:S}
Let $\displaystyle S(x,y) = \sum_{i=0}^{\infty}\sum_{j=0}^{\infty}|S_i(j)|x^iy^j$. Then 
\[S(x,y) = \frac{1-y}{1-x-y-xy}.\]
\end{theorem}
\begin{proof}
We know that $|S_n(d)|=|B_n(d)|-|B_n(d-1)|$. From that, we get this generating function
$\displaystyle S(x,y)=B(x,y)-yB(x,y)=  \frac{1-y}{1-x-y-xy}.  $
\end{proof}

We now look at how $B_n(d)$ and $S_n(d)$ vary with $n$ or $d$ fixed. To do so, we derive single variable generating functions, which can be done in the same manner. Consider $B_d(x) = \sum_{i=0}^\infty |B_i(d)|x^i$ and $B_n(y) = \sum_{i=0}^\infty |B_n(i)|y^i$. 
\begin{theorem} \label{theorem:genball}
If $\displaystyle B_d(x) = \sum_{i=0}^\infty |B_i(d)|x^i$ and $\displaystyle B_n(y) = \sum_{i=0}^\infty |B_n(i)|y^i$, then
\begin{displaymath}B_d(x) = \frac{(1+x)^d}{(1-x)^{d+1}}\text{ and}\end{displaymath}
\begin{displaymath}B_n(y) = \frac{(1+y)^n}{(1-y)^{n+1}}.\end{displaymath}
\end{theorem}
\begin{proof}
By Lemma \ref{lemma:recursion}, we have
\begin{displaymath}B_d(x) = xB_d(x) + B_{d-1}(x) + xB_{d-1}(x).\end{displaymath} 
We know that $B_0(x) = \frac{1}{1-x}$ and
$\displaystyle B_d(x) = \frac{1+x}{1-x}B_{d-1}(x)$, therefore by induction we have \begin{displaymath}B_d(x) = \frac{(1+x)^d}{(1-x)^{d+1}}.\end{displaymath}

Since $B(x,y)$ is symmetric over $x$ and $y$, we know that
$\displaystyle B_n(y) = \frac{(1+y)^n}{(1-y)^{n+1}}.$\qedhere
\end{proof}

We can perform similar derivations for $S_d(x)$ and $S_n(y)$.
\begin{theorem} \label{theorem:genshell}Let $\displaystyle S_d(x) = \sum_{i=0}^\infty |S_i(d)|x^i$ and $\displaystyle S_n(y) = \sum_{i=0}^\infty |S_n(i)|y^i$. Then, for $d > 0$ and $n \ge 0$,
\begin{displaymath}S_d(x) = \frac{2x(1+x)^{d-1}}{(1+x)^{d+1}}\end{displaymath} 
\begin{displaymath}\text{and } S_n(y) = \frac{(1+y)^n}{(1-y)^n}.\end{displaymath}
\end{theorem}
\begin{proof}If $d=0$, then $S_d(x) = \frac{1}{1-x}$. Otherwise
\begin{displaymath}S_d(x) = B_d(x) - B_{d-1}(x) = \frac{(1+x)^d}{(1-x)^{d+1}}-\frac{(1+x)^{d-1}}{(1-x)^{d}}= \frac{2x(1+x)^{d-1}}{(1+x)^{d+1}}.\end{displaymath}
Similarly, $\displaystyle S_n(y) = B_n(y)-xB_n(y) = (1-y)\frac{(1+y)^n}{(1-y)^{n+1}} = \frac{(1+y)^n}{(1-y)^n}.$ \qedhere

\end{proof}

From the formulas in Theorems~\ref{theorem:shell} and \ref{theorem:ball} and from generating functions in Theorems~\ref{theorem:genball} and \ref{theorem:genshell}, we found $|S_n(d)|$ and $|B_n(d)|$ correspond to Sequences \href{http://oeis.org/A265014}{OEIS A265014} and \href{http:/oeis.org/A008288}{OEIS A008288} in Sloane's Online Encyclopedia of Integer Sequences. Both sequences have applications to cellular automata, and $|B_n(d)|$ correspond to the well-studied Delannoy numbers. This result was previously independently derived by Zaitsev in 2017. The $(m,n)$th Delannoy number describes the number of paths from the origin to the integer point $(m,n)$, where, from each point $(i,j)$, $(i+1,j)$, $(i,j+1)$, and $(i+1,j+1)$ can be reached.

The fact that an $n$-dimensional ball of radius $d$ has the same number of points as a $d$-dimensional ball of radius $n$ is surprising, though an automatic consequence of the symmetry of the generating function $B(x,y)$, and suggests a simple bijection between the two sets. To identify such a bijection, we first introduce some notation. 
\begin{definition}
Let a \emph{tuple-sequence} be a list of signed tuples $[\pm (n_1, d_1), \pm(n_2, d_2) \ldots ]$, where $n_i$ and $d_i$ are positive integers for all $i$. Let the \emph{dimension sum} of a tuple-sequence be defined to be the sum of the first terms of the tuples, while the \emph{distance sum} is the sum of the second terms. 
Let $T_{n,d}$ denote the set of tuple-sequences with a dimension sum less than or equal to $n$ and a distance sum less than or equal to $d$.
\end{definition}

Every tuple-sequence describes a unique two-colored lattice path starting at $(0,0)$, with steps that increase both coordinates. Thus the set $T_{n,d}$ corresponds to the set of lattice paths that end at a point $(m,c)$ such that $m \le n$ and $c\le d$.

The following example demonstrates the bijections described above. 
\begin{example}
In $B_4(3) \xleftrightarrow{\phi} T_{4,3} \xleftrightarrow{\psi} T_{3,4} \xleftrightarrow{\phi^{-1}} B_3(4)$: 
\begin{displaymath}(2, 0,-1,0) \leftrightarrow (+(1,2), -(2,1)) \leftrightarrow (+(2,1),-(1,2)) \leftrightarrow (0, 1, -2)\end{displaymath}
\begin{displaymath}(-1, 0, 1, -1)\leftrightarrow(-(1,1), +(2,1),-(1,1))\leftrightarrow(-(1,1),+(1,2),-(1,1))\leftrightarrow(-1,2,-1)\end{displaymath}
\begin{displaymath}(0,0,0,0)\leftrightarrow()\leftrightarrow()\leftrightarrow(0,0,0)\end{displaymath}
In the first example, we consider the vertex with coordinates $(2,0,-1,0)$, which is inside $B_4(3)$. First, we convert it to a tuple-sequence. Since the first nonzero coordinate is 2 and is the first coordinate, we get the signed tuple $+(1,2)$. Since the second nonzero coordinate is -1, which is two coordinates later, we get the signed tuple $-(2, 1)$. Then we apply $\psi$, which flips all the tuples. Finally, we convert them back into a vertex in $B_3(4)$. Since the first tuple is $+(2,1)$, we get the that the second coordinate is 1. Since the next tuple is $-(1,2)$, we know that the next coordinate is -2. All others are 0. Thus, we get to $(0, 1, -2)$.

\end{example}

\begin{theorem} \label{thm:trivial}
There is a bijection between the sets $B_n(d)$, $B_d(n)$, $T_{n,d}$, and $T_{d,n}$.
\end{theorem}

\begin{proof} 
 Let $\phi$ be a mapping between points in $\ZZ^n$ and tuple-sequences. Consider point $p$ in $B_n(d)$ with coordinates $(x_1, \ldots , x_n)$ such that \begin{displaymath}\sum_{i=1}^n|x_i| \leq d.\end{displaymath} The map $\phi$ sends $p$ to the tuple-sequence $t$, such that each second term of a tuple of $t$ corresponds to the magnitude of a nonzero value of $p$, each first term corresponds to the distance between that value and the previous nonzero value (or the beginning of the sequence, if there is no first value), and the sign of that tuple corresponds to the sign of that value, i.e., 
\begin{align*}
&(\pm (n_1, d_1), \pm(n_2, d_2), \ldots)
\\
=&\phi((\textrm{[$n_1-1$ 0s]}, \pm d_1, \textrm{[$n_2-1$ 0s]}, \pm d_2, \ldots, d_k, 0, 0, 0, \ldots))
\end{align*}
The dimension sum of $t$ is the position of the last nonzero term of $p$. Since $p$ has $n$ terms, that must be less than or equal to $n$. The distance sum of $t$ is equal to the sum of the magnitudes of the values of $p$, which is the distance from the origin to $p$. Since $p$ is within a distance $d$ from the origin, the distance sum of $t$ must be less than or equal to $d$. Thus, $t$ is in $T_{n,d}$. For similar reasons, for any tuple-sequence $s$ in $T_{n,d}$, $\phi^{-1}(s)$ is in $B_n(d)$. Since $\phi$ is a one to one invertible function, it must be a bijection.

Let $\psi$ be the bijection between $T_{n,d}$ and $T_{d,n}$ defined by reversing the order of all the signed tuples in them. Since the composition of bijections is a bijection, we see that $\phi \circ \psi \circ \phi^{-1}$ defines a bijection between $B_n(d)$ and $B_d(n)$.
\end{proof}

\section{Lower Bound on Densities for Infinite Grids}
\label{sec:lbound}

In this section, we calculate lower bounds for the density of a ($t,r$) broadcast pattern which dominates an infinite grid or finite grid of any dimension. To begin, we note that since any reception provided to a vertex beyond $r$ is always wasted, it is often useful to ignore it. The \emph{unwasted reception} $r^*(d, v)$ provided to a vertex $v$ from a broadcast $d$ is defined to be $min(r(d, v), r)$. Define the coverage $c(d)$ of a broadcast $d$ to be $\sum_{v \in N_t(d)}r^*(d, v)$.
Given a transmission strength $t$, a desired reception $r$, and a graph $G$, we define $C_{t,r}(G)$ to be the maximum coverage provided by any vertex in $G$.  

An example is shown in Figure \ref{fig3}. In that figure, the broadcast shown in blue receives 4 reception, but since $r = 3$ the unwasted reception is only 3. The black vertices all receive $3$ reception. Both the blue and black vertices are dominated. The red and pink vertices all receive less than $r$ reception, so they receive an equal amount of unwasted reception and are not dominated. Since there are $12$ vertices receiving 1 reception, $8$ receiving 2, and 4 receiving 3 unwasted reception, the coverage provided by the center is $40$, so $C_{4,3}(\ZZ^2) = 40$.

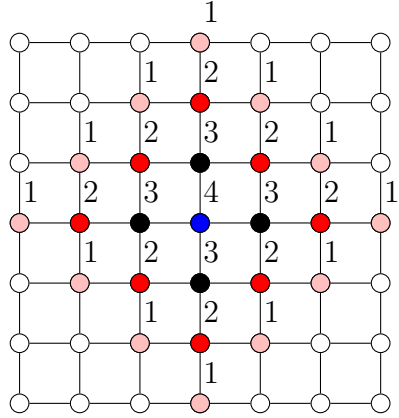
\begin{figure}[ht!]
\centering
\begin{tikzpicture}[scale =.8]
 \tikzstyle{ghost node}=[draw=none]
\tikzset{white node/.style={circle,draw=black, inner sep=2.5}}
\tikzset{red node/.style={circle,draw=black, fill=red, inner sep=2.5}}
\tikzset{black node/.style={circle,draw=black, fill=black, inner sep=2.5}}
\tikzset{pink node/.style={circle,draw=black, fill=pink, inner sep=2.5}}
\tikzset{gray node/.style={circle,draw=black, fill=gray, inner sep=2.5}}
\tikzset{blue node/.style={circle,draw=black, fill=blue, inner sep=2.5}}

 \foreach \x in {1, 2, 3,4,5,6,7}
    \foreach \y in {1, 2, 3,4,5,6,7} 
       \node [white node] (\x\y) at (1*\x,1*\y){};
       \node [label={[xshift=0.15cm, yshift=0cm]4}] [blue node] (44) at (1*4,1*4){};
       \node[label={[xshift=0.15cm, yshift=0cm]1}] [pink node] (47) at (1*4,1*7){};
       \node[label={[xshift=0.15cm, yshift=0cm]1}] [pink node] (56) at (1*5,1*6){};
       \node[label={[xshift=0.15cm, yshift=0cm]1}] [pink node] (65) at (1*6,1*5){};
       \node[label={[xshift=0.15cm, yshift=0cm]1}] [pink node] (74) at (1*7,1*4){};
       \node[label={[xshift=0.15cm, yshift=0cm]1}] [pink node] (63) at (1*6,1*3){};
       \node[label={[xshift=0.15cm, yshift=0cm]1}] [pink node] (52) at (1*5,1*2){};
       \node[label={[xshift=0.15cm, yshift=0cm]1}] [pink node] (41) at (1*4,1*1){};
       \node[label={[xshift=0.15cm, yshift=0cm]1}] [pink node] (32) at (1*3,1*2){};
       \node[label={[xshift=0.15cm, yshift=0cm]1}] [pink node] (23) at (1*2,1*3){};
       \node[label={[xshift=0.15cm, yshift=0cm]1}] [pink node] (14) at (1*1,1*4){};
       \node[label={[xshift=0.15cm, yshift=0cm]1}] [pink node] (25) at (1*2,1*5){};
       \node[label={[xshift=0.15cm, yshift=0cm]1}] [pink node] (36) at (1*3,1*6){};
       \node[label={[xshift=0.15cm, yshift=0cm]2}] [red node] (46) at (1*4, 1*6){};
       \node[label={[xshift=0.15cm, yshift=0cm]2}] [red node] (55) at (1*5, 1*5){};
       \node[label={[xshift=0.15cm, yshift=0cm]2}] [red node] (64) at (1*6, 1*4){};
       \node[label={[xshift=0.15cm, yshift=0cm]2}] [red node] (53) at (1*5, 1*3){};
       \node[label={[xshift=0.15cm, yshift=0cm]2}] [red node] (42) at (1*4, 1*2){};
       \node[label={[xshift=0.15cm, yshift=0cm]2}] [red node] (33) at (1*3, 1*3){};
       \node[label={[xshift=0.15cm, yshift=0cm]2}] [red node] (24) at (1*2, 1*4){};
       \node[label={[xshift=0.15cm, yshift=0cm]2}] [red node] (35) at (1*3, 1*5){};
       \node[label={[xshift=0.15cm, yshift=0cm]3}] [black node] (45) at (1*4, 1*5){};
       \node[label={[xshift=0.15cm, yshift=0cm]3}] [black node] (54) at (1*5, 1*4){};
       \node[label={[xshift=0.15cm, yshift=0cm]3}] [black node] (43) at (1*4, 1*3){};
       \node[label={[xshift=0.15cm, yshift=0cm]3}] [black node] (34) at (1*3, 1*4){};

  \foreach \x in {1,2,3,4,5,6,7}
    \foreach \y  [count=\yi from 2] in {1,2,3,4,5,6} 
      \path[] (\x\y)edge(\x\yi)(\y\x)edge(\yi\x);
\end{tikzpicture} 
\caption{A figure showing the $(4, 3)$ coverage provided by the center vertex in $\ZZ^2$. } \label{fig3}
\end{figure}

\begin{lemma}\label{lemma:coverage}
Let $C_{t,r}(\ZZ^n)$ be the amount of unwasted reception provided by a $(t,r)$ broadcast placed at the origin of $\ZZ^n$, and let $|S_n(d)|$ describe the number of points a distance exactly $d$ away from the origin in $\ZZ^n$. Then 
\begin{displaymath}C_{t,r}(\ZZ^n) = \sum_{d = t - r + 1}^{t - 1}(t - d)|S_n(d)| + r\sum_{d = 1}^{t - r}|S_n(d)| + r.\end{displaymath}
\end{lemma}

\begin{proof}
Each broadcast provides $r$ unwasted reception to itself, and it also provides $r$ unwasted reception to each vertex a distance $t - r$ or less away. Thus it supplies a total unwasted reception to these vertices of 
\begin{displaymath}r + r\sum_{d = 1}^{t - r}|S_n(d)|.\end{displaymath}
For each vertex a distance $t - r < d < t$ away, the broadcast provides $t - d$ reception, none of which is wasted. The total reception provided to these vertices is
\begin{displaymath}\sum_{d = t - r + 1}^{t - 1}(t - d)|S_n(d)|.\end{displaymath}
Summing these, we get
\begin{align*}C_{t,r}(\ZZ^n) 
& = \sum_{d = t - r + 1}^{t - 1}(t - d)|S_n(d)| + r\sum_{d = 1}^{t - r}|S_n(d)| + r.
\end{align*} 
\end{proof}

Note that when $n$ is fixed, $C_{t,r}(\ZZ^n)$ becomes a degree $n+1$ polynomial in $t$ and $r$, as seen in Table \ref{tbl:Ctr}.

   \begin{table}[ht!]
        \centering
\begin{tabular}{|c|c| } \hline 
 $n$ & $C_{t,r}$ \\ \hline
    1 & $2tr - r^2$\\
     2 &  $\frac{2r^3}{3} - 2r^2 t + 2rt^2 + \frac{r}{3}$ \\
     3 &  $-\frac{r^4}{3} + \frac{4r^3t}{3} - 2r^2t^2 - \frac{2r^2}{3} + \frac{4rt^3}{3} + \frac{4rt}{3}$  \\
     4 & $\frac{2r^5}{15} - \frac{2r^4t}{3} + \frac{4r^3t^2}{3} + \frac{2r^3}{3} - \frac{4r^2t^3}{3} - 2r^2t + \frac{2rt^4}{3} + 2rt^2 + \frac{r}{5}$  \\ \hline
\end{tabular}    
\caption{Total unwasted reception provided by a $(t,r)$ broadcast in an $n$ dimensional grid}\label{tbl:Ctr}
\end{table} 

We are now ready to prove a lower bound on the $(t,r)$ broadcast domination of a finite grid $G$.

\begin{theorem} \label{thm:lowerbound}
Let $G$ be a finite $n$-dimensional grid graph with vertex set $V$. Then
\begin{displaymath}\gamma_{t,r}(G) \geq \frac{r|V|}{C_{t, r}(\ZZ^n)}.\end{displaymath}
\end{theorem}

\begin{proof}
Since $G$ is a subgraph of $\ZZ^n$, the maximum amount of coverage a vertex can provide is at most $C_{t,r}(\ZZ^n)$. The total (minimum) reception needed to dominate the graph is equal to the product of $r$ and the number of vertices in the graph. Since each vertex can provide at most $C_{t,r}(\ZZ^n)$ unwasted reception, we get
\begin{displaymath}\gamma_{t,r}(G)C_{t,r}(\ZZ^n) \geq r|V|.\end{displaymath}
Dividing both sides by $C_{t,r}(\ZZ^n)$ gives a lower bound for $\gamma_{t,r}(G)$.
\end{proof}

\section{Computer Implementation}  
\label{sec:alg}

To find more exact values for the density of $(t, r)$ dominating patterns on the infinite grid, we built a program to test various values. This program uses tower notation, first defined in \cite{drews2019optimal}. A tower set is denoted by $T(d, e)$, and is defined to be the set of vertices in $\mathbb{Z}^2$ with coordinates $(md+ne, n)$, with $m, n \in \mathbb{Z}$. By its definition $T(d, e)$ has density $\frac{1}{d}$. For various values of $(t, r)$, our program goes through various $T(d, e)$ and returns the one with the highest possible value of $d$. In the pseudocode below, $\text{IsBroadcast}(t, r, d, e)$ determines whether placing $(t, r)$ broadcasts at $T(d, e)$ dominates the grid, then MaxPotentialD($t, r$) finds the highest possible value of $d$ based on the formula from Lemma \ref{lemma:coverage}:  
\begin{displaymath}\textup{MaxPotentialD}(t,r) = \left  \lfloor \frac{1}{\gamma_{t,r}(G)} \right 
\rfloor \leq \frac{C_{t, r}(\ZZ^n)}{r|V|}.\end{displaymath}

\begin{algorithm}[ht!]
  \caption{Optimal (t, r) Broadcast Dominating Pattern}
\begin{algorithmic}
 \Procedure{MinDensity}{$t$, $r$}\;
 $d$ = MaxPotentialD($t, r$)\;
 \While{$d \geq 1$}{
  $e \leftarrow 1$\;
  \While{$e < d$}{
   \If{Rebroadcast($t, r, d, e$)}{
   return d\;
   }$e++$
   }$d--$}
  \EndProcedure
  \Procedure{IsBroadcast}{$t,r,d,e$}\;
  Array a = new Array[d]\;
  \For{$i = 0$; $i<t$; $i++$}{
  a[i]=a[-i]=t-i\;
  }
  \For{$i$;$t-r+1 \le i \le d/2$}{
  rcpt = a[i]\;
 \For{$j=1$; $j<t$; $j++$}{
 \If{a[(i+je) mod d] $>$ j}{
 rcpt += a[(i+je) mod d]-j\;
 }
 \If{a[(i-je) mod d] $>$ j}{
 rcpt += a[(i-je) mod d]-j\;
 }
 } 
 \If{rcpt $<$ r}{return false\;}
 }
 return true\;
 \EndProcedure
 \end{algorithmic}
 \end{algorithm}

Since the set of broadcasts is invariant under the transformations
\begin{displaymath}(x,y) \to (x+md+ne, y+n),\ m, n \in \ZZ\end{displaymath}
\begin{displaymath}(x, y) \to (-x, y)\end{displaymath}
only the vertices from $(0, 0)$ to $(\lfloor d/2 \rfloor, 0)$ need to be checked. Each increase of $y$ by 1 is equivalent to a shift by $e$ in the range $(0, 0)$ to $(d-1, 0)$. Thus, the algorithm returns the reciprocal of the minimum density of a $(t, r)$ broadcast dominating tower set.
 
\begin{example}
We illustrate the algorithm when  $(t, r) = (4, 2)$. The algorithm proceeds as follows: 

 For $(t, r)$ = $(4, 2)$, a single broadcast provides $38$ unwasted reception, so the maximum number of vertices that could be dominated by it is $19$. The algorithm iterates over variables $d$ and $e$, with $d$ starting at $19$ and going down to $1$, and $e$ starting at $1$ and going up to $d - 1$. The algorithm simulates an infinite grid with broadcasts placed at points described by $T(d, e)$. This is the set of points $(x, y)$ such that $x = ye$ mod $d$. For example, $T(18, 5)$ contains the points $(0, 0), (18, 0), (5, 1), (36, 0), (23, 1), (10, 2), \ldots$
    
     Since $T(d, e)$ is symmetric around each broadcast, only the $18$ vertices between the lattice points $(0, 0)$ and $(17, 0)$, inclusive, must be checked.
     We now add the receptions from only the broadcasts on the $x$-axis. The corresponding receptions are: \begin{displaymath}[4, 3, 2, 1, 0,0,0,0,0,0,0,0,0,0,0,1, 2, 3].\end{displaymath}
     Since each row is shifted by $5$ from the adjacent rows, to account for the other broadcasts we need to shift the reception pattern we got above by $5$, decrease all nonzero terms by $1$ since they are now further away, and add them to the receptions. So the reception received from the row above the $x$-axis is \begin{displaymath}[0,0,0,1,2,3,2,1,0,0,0,0,0,0,0,0,0,0].\end{displaymath} Table \ref{tbl:tower} below shows the reception received from each row of broadcasts. 

    \begin{table}[ht!]
    \label{tab:algex}
        \centering
        \begin{tabular}{c|c|c|c|c|c|c|c|c|c|c|c|c|c|c|c|c|c|c}
             &0&1&2&3&4&5&6&7&8&9&10&11&12&13&14&15&16&17 \\ \hline
             3&0&0&0&0&0&0&0&0&0&0&0&0&0&0&0&1&0&0 \\ 
             2&0&0&0&0&0&0&0&0&0&1&2&1&0&0&0&0&0&0 \\
             1&0&0&0&1&2&3&2&1&0&0&0&0&0&0&0&0&0&0 \\
             0&4&3&2&1&0&0&0&0&0&0&0&0&0&0&0&1&2&3 \\
             -2&0&0&0&0&0&0&0&1&2&1&0&0&0&0&0&0&0&0 \\
             -1&0&0&0&0&0&0&0&0&0&0&0&1&2&3&2&1&0&0 \\
             -3&0&0&0&1&0&0&0&0&0&0&0&0&0&0&0&0&0&0 \\ \hline
             Sum &4&3&2&3&2&3&2&2&2&2&2&2&2&3&2&3&2&3
        \end{tabular}
        \caption{The reception received by the vertex $(x, 0)$ from the broadcasts on row $y$, where $x$ is represented by the horizontal values and $y$ by the vertical. }\label{tbl:tower}
    \end{table}
    Since the sum of the receptions received from all rows is at least $2$ at every vertex, $T(18,5)$ dominates the grid under $(4,2)$ broadcast domination. Since all tower sets with lower densities have been shown to not dominate, this is the most efficient tower set for dominating the grid.
\end{example}

A similar algorithm is contained in \cite{drews2019optimal}. However, our algorithm not only works for 3-dimensional grids, it also runs faster (by a factor about about 415) when $t,r \le 10$. For instance, when implemented in Cocalc 
it takes our algorithm 1.39 seconds to calculate optimal broadcast densities for all $t,r \le 10$, while it takes the aforementioned algorithm 9 minutes and 36 seconds. The code used in Cocalc can be found \href{https://github.com/TomShlomi/TowerAlgorithmComparison}{here}.

Table \ref{table:1} contains the minimum densities of tower broadcasts for $(t,r) \leq (9,9)$.  A more complete set of data is found \href{https://docs.google.com/spreadsheets/d/1F2MNDzp_j10FXPlwU3vpCwrFzqyXJzAIyZdXN4pQPFk/edit?usp=sharing}{here}, and the code used is found \href{https://github.com/TomShlomi/t-r-Broadcast-Density/tree/Submission0}{here}. This expands on the results found by \cite{drews2019optimal}, and agrees with all values found there. The actual density seems to always be close to the minimum density, and appears to limit to 0 as $t$ grows. When $r = 1$, the density equals the minimum density. 

\begin{center}
\begin{table}[ht!]
\begin{tabular}{|c|c|c|c|c|c|c|c|c|c|c|} \hline
& $r = 1$ & $r = 2$ & $r = 3$ & $r = 4$ & $r = 5$ & $r = 6$ & $r = 7$ & $r = 8$ & $r = 9$ \\ \hline
$t=1$&1&&&&&&&&\\
$t=2$&5&3&&&&&&&\\
$t=3$&13&8&$5$&&&&&&\\
$t=4$&25&18&13&10&&&&&\\ 
$t=5$&41&32&25&18&14&&&&\\
$t= 6$&61&50&41&34&26&22&&&\\
$t=7$&85&72&61&50&42&36&29&&\\
$t=8$&113&98&85&74&62&54&43&39&\\
$t=9$&145&128&113&98&86&76&65&58&49\\  
     \hline
\end{tabular}
\caption{The reciprocal of the minimum density for small $t$ and $r$.} \label{table:1}
\end{table}
\end{center} 

\acknowledgements
The author would like to thank his mentor Dr. Erik Insko, who introduced him to this field and guided him throughout the research and writing of this paper. He would also like to thank Dr. Katie Johnson, Dr. Shaun Sullivan, and Dr. Pamela Harris, for their helpful comments and advice. Finally, he thanks an anonymous reviewer for their helpful comments that improved the quality of this manuscript.

\nocite{*}
\bibliographystyle{abbrvnat}
\bibliography{Paper}

\begin{thebibliography}{7}
\providecommand{\natexlab}[1]{#1}
\providecommand{\url}[1]{\texttt{#1}}
\expandafter\ifx\csname urlstyle\endcsname\relax
  \providecommand{\doi}[1]{doi: #1}\else
  \providecommand{\doi}{doi: \begingroup \urlstyle{rm}\Url}\fi

\bibitem[Blessing et~al.(2015)Blessing, Johnson, Mauretour, and
  Insko]{blessing2015t}
D.~Blessing, K.~Johnson, C.~Mauretour, and E.~Insko.
\newblock On (t, r) broadcast domination numbers of grids.
\newblock \emph{Discrete Applied Mathematics}, 187:\penalty0 19--40, 2015.

\bibitem[Bre{\v{s}}ar et~al.(2012)Bre{\v{s}}ar, Dorbec, Goddard, Hartnell,
  Henning, Klav{\v{z}}ar, and Rall]{BreDorGodHarHenKlaRal12}
B.~Bre{\v{s}}ar, P.~Dorbec, W.~Goddard, B.~L. Hartnell, M.~A. Henning,
  S.~Klav{\v{z}}ar, and D.~F. Rall.
\newblock Vizing's conjecture: a survey and recent results.
\newblock \emph{Journal of Graph Theory}, 69\penalty0 (1):\penalty0 46--76,
  2012.

\bibitem[Drews et~al.(2019)Drews, Harris, and Randolph]{drews2019optimal}
B.~F. Drews, P.~E. Harris, and T.~W. Randolph.
\newblock Optimal (t, r) broadcasts on the infinite grid.
\newblock \emph{Discrete Applied Mathematics}, 255:\penalty0 183--197, 2019.

\bibitem[Farina and Grez(2016)]{farina2016new}
M.~Farina and A.~Grez.
\newblock New upper bounds on the distance domination numbers of grids.
\newblock \emph{Rose-Hulman Undergraduate Mathematics Journal}, 17\penalty0
  (2):\penalty0 7, 2016.

\bibitem[Haynes et~al.(2013)Haynes, Hedetniemi, and
  Slater]{haynes2013fundamentals}
T.~W. Haynes, S.~Hedetniemi, and P.~Slater.
\newblock \emph{Fundamentals of domination in graphs}.
\newblock CRC press, 2013.

\bibitem[Vizing(1968)]{Viz68}
V.~G. Vizing.
\newblock Some unsolved problems in graph theory.
\newblock \emph{Russian Mathematical Surveys}, 23\penalty0 (6):\penalty0 125,
  1968.

\bibitem[Zaitsev(2017)]{Zait17}
D.~A. Zaitsev.
\newblock A generalized neighborhood for cellular automata.
\newblock \emph{Theoretical Computer Science}, 666:\penalty0 21--35, 2017.

\end{thebibliography}
\label{sec:biblio}

\end{document}